\documentclass[11pt,a4paper,twoside]{article}
\usepackage{graphics}
\usepackage{amssymb}
\usepackage{amsmath}
\usepackage{mathrsfs}
\usepackage{makeidx}
\usepackage{color}
\usepackage{graphicx}
\usepackage{subfigure}
\usepackage{multicol}
\usepackage{multirow}
\usepackage{float}
\usepackage{booktabs}
\usepackage{epsfig}
\usepackage{multirow}
\usepackage{epstopdf}
\usepackage[colorlinks,
linkcolor=red,
anchorcolor=blue,
citecolor=blue
]{hyperref} 
\usepackage{caption}
\usepackage{cases}
\usepackage{algorithm,algorithmic}

\usepackage{amsthm} 
\theoremstyle{definition} 
\newtheorem{theorem}{Theorem} 

\newtheorem{definition}{Definition}[section]
\newtheorem{lemma}{Lemma}[section] 
\usepackage{latexsym, cite, bm, amsthm, amsmath}
\usepackage{enumerate}
\usepackage{marginnote}
\usepackage{epstopdf}
\usepackage{pifont}
\usepackage{cleveref}

\def \[{\begin{equation}}
	\def \]{\end{equation}}

\newtheorem{remark}{Remark}[section]

\newtheorem{exam}{Example}[section]

\numberwithin{equation}{section}

\title{A fixed-time stable dynamical model for solving EVLCPs}

\usepackage[marginal]{footmisc}
\usepackage{authblk}
\author[a]{Yufei Wei\thanks{Email address: 707005495@qq.com.}}
\author[a]{Shiping Lin\thanks{Email address:  2660654016@qq.com. }}
\author[a]{Cairong Chen\thanks{Corresponding author. Email address: cairongchen@fjnu.edu.cn.}}
\author[b]{Dongmei Yu\thanks{Email address: yudongmei1113@163.com.}}
\author[c]{Deren Han\thanks{Email address: handr@buaa.edu.cn.}}
\affil[a]{School of Mathematics and Statistics, FJKLMAA and Center for Applied Mathematics of Fujian Province, Fujian Normal University, Fuzhou, 350117, P.R. China.}
\affil[b]{Institute for Optimization and Decision Analytics, Liaoning Technical University, Fuxin, 123000, P.R. China.}
\affil[c]{LMIB of the Ministry of Education, School of Mathematical Sciences, Beihang University, Beijing, 100191, P.R. China.}

\begin{document}
	
	\date{\today}
	\maketitle
	
	\begin{quote}
		{\bf Abstract:} A fixed-time stable dynamical system for solving the extended vertical linear complementarity problem (EVLCP) is developed. The system is based on  the reformulation of EVLCP as a special case of a new kind of generalized absolute value equations. Some properties of the new kind of generalized absolute value equations are explored which are useful for developing a fixed-time stable dynamical system for solving it. Without using any smoothing technique, we develop a dynamical system for solving the new kind of generalized absolute value equations and prove its fixed-time stability. The model is applicable for solving EVLCP. As two by-products, a new condition which guarantees the unique solvability of EVLCP and a new error bound of EVLCP are provided. Numerical results are given to demonstrate our claims.	
		
		{\small
			
		}
		
	\end{quote}

\section{Introduction}
Let $A_j\in \mathbb{R}^{n\times n}$ and $q_j\in\mathbb{R}^n (j = 0,1,2,\ldots, s, s\ge 1)$ be known matrices and vectors. The extended generalized order linear complementarity problem \cite{gosz1994} is to find a vector  $x\in \mathbb{R}^n$ such that
	\begin{equation}\label{eq:evlcp}
		\min \{A_0x +q_0, A_1x + q_1,\ldots, A_s x + q_s\} = 0,
	\end{equation}
	where $\min$ is the component minimum operator. As shown in \cite{gosz1994}, when $A_0$ is invertible, the problem \eqref{eq:evlcp} is equivalent to the following vertical linear complementarity problem (VLCP) proposed in \cite{coda1970}: to find a vector  $x\in \mathbb{R}^n$ such that
	\begin{equation}\label{eq:vlcp}
		\begin{aligned}
			w &= q+ Nx,\\
			w&\ge 0, \quad x\ge 0,\\
			x_j \prod_{i = 1}^{p_j}w_i^j &= 0 \quad (j = 1,2,\ldots,\ell),
		\end{aligned}
	\end{equation}
	where
	$$
	N = \left[\begin{array}{c}N^1\\N^2\\\vdots\\N^\ell\end{array}\right], \quad w = \left[\begin{array}{c} w^1\\w^2\\\vdots\\w^\ell\end{array}\right]\quad \text{and} \quad q = \left[\begin{array}{c} q^1\\q^2\\\vdots\\q^\ell \end{array}\right]
	$$
	with the $j$-th block $N^j$ having order $p_j \times n$, $w^j = N^j x + q^j$ and $\sum_{j=1}^{\ell} p_j = n$. In \cite{szgo1995}, the problem~\eqref{eq:evlcp} is called as the extended VLCP (EVLCP). When $s =1$, $A_0 = I$ and $q_0 = 0$, EVLCP reduces to the standard linear complementarity problem (LCP) \cite{cops1992} which has been well studied in the literature. EVLCP~\eqref{eq:evlcp} can also be reformulated as an implicit complementarity problem \cite{pang1981}. In addition, VLCP~\eqref{eq:vlcp} is equivalent to a nonlinear complementarity problem, a piecewise linear system of equations, a multiple objective programming problem, and a variational inequality problem \cite{ebie1995}.
	
EVLCP has many applications such as the nonlinear network \cite{fuku1972}, the singular control problems in bounded intervals \cite{sun1987,sun1989} and the real option problems \cite{naak2008}. VLCP also has many applications, including the generalized Leontief input-output model \cite{ebko1993} and the generalized bimatrix game \cite{gosz1996}.   For extensions of EVLCP and their theory, algorithms and applications, see \cite{isgo1993,hozq2022,scmo1995} and the references therein.
	
Over the past few decades, there have been numerous investigations on both theoretical and numerical aspects for EVLCP and VLCP. Theoretically, conditions for the existence and uniqueness of solutions to EVLCP or VLCP can be found in, e.g., \cite{szgo1995,gosz1994,mons1996,ebie1995,ebko1992,mang1979} and the references therein. The error bounds for EVLCP can be found in \cite{wazl2021,wuwa2022,zhcx2009}.  Numerically, several algorithms for solving EVLCP or VLCP are available. For example, Cottle and Dantzig \cite{coda1970} extend Lemke's algorithm \cite{lemk1965} to solve VLCP~\eqref{eq:vlcp}. Based on the so-called aggregation smoothing function, smoothing Newton methods \cite{qili1999,zhga2003} and a non-interior continuation method \cite{peli1999}  are proposed  for solving EVLCP and VLCP. Ebiefung et al. \cite{efjk2022} develop a block principal pivoting algorithm for VLCP. Mezzadri \cite{mezz2022} introduces a modulus-based formulation and develops modulus-based matrix splitting methods for VLCP. However, Mezzadri's reformulation involves some auxiliary variables.  He and Vong \cite{hevo2022}, without using any auxiliary variable, propose a new kind of modulus-based matrix splitting methods for VLCP. Yu et al. \cite{ywch2024} propose a scalable relaxation two-sweep modulus-based matrix splitting method for VLCP. The equivalent modulus-based reformulation of VLCP proposed in \cite{ywch2024} differs from that of \cite{hevo2022} while both of them are independent of any  auxiliary variable. Recently, modulus-based formulations and methods for VLCP have received increasing attention and have led to a series of research efforts; see, e.g., \cite{zzlv2023,wali2024,gzlz2025} to name just a few. Another kind of methods for solving EVLCP or VLCP are the projected-type methods \cite{mega2022,liwu2024,cays2023}. For more numerical algorithms, one is referred to \cite{mone1996} and the references therein.
	
Unlike the discrete iterative methods mentioned above, a useful way for solving EVLCP~\eqref{eq:evlcp} can be the continuous dynamical models. Though the continuous dynamical models have been used for solving many mathematical problems \cite{xlwc2024,liqq2004}, the dynamical models for solving EVLCP~\eqref{eq:evlcp} are rare. To the best of our knowledge, the neural network proposed in \cite{hozq2022} can be used to solve EVLCP~\eqref{eq:evlcp}. The model based on the  aggregation smoothing function and the equilibrium point is proved to be asymptotically stable or exponentially stable under certain conditions \cite{hozq2022}. However, in real applications, the finite-stability is more practical. The notion of finite-time stability of an equilibrium is introduced in \cite{bhbe2000} and then extended to the fixed-time stability \cite{poly2011} whose settling time is independent of initial conditions. This motivates us to develop a fixed-time stable dynamical system for solving EVLCP~\eqref{eq:evlcp}.
	
Our method is based on the reformulation of EVLCP~\eqref{eq:evlcp} as a special case of a new kind of generalized absolute value equations. Some properties of the new kind of generalized absolute value equations are explored which are useful for developing a fixed-time stable dynamical system for solving it. Without using any smoothing technique, we develop a dynamical system for solving the new kind of generalized absolute value equations and prove its fixed-time stability. The model is applicable for solving EVLCP~\eqref{eq:evlcp}. As two by-products, a new condition which guarantees the unique solvability of EVLCP~\eqref{eq:evlcp} for any $q_i$ and a new error bound of EVLCP~\eqref{eq:evlcp} are given. Numerical results are given to demonstrate the effectiveness of our model.   	
	
The rest of this paper is organized as follows. In Section~\ref{sec:pre} we introduce some results about dynamical systems. In Section~\ref{sec:dm}, we develop a new kind of generalized absolute value equations, explore its properties and develop a dynamical system for solving it. In addition, the fixed-time stability of the proposed model is proved.  The results can be applied to EVLCP~\eqref{eq:evlcp}. In Section~\ref{sec:ne}, numerical results are given to demonstrate our claims. Conclusion remarks are given in Section~\ref{sec:con}.
	
\textbf{Notations.} We use $\mathbb{R}^{n \times n}$ to denote the set of all $n \times n$ real matrices, while $\mathbb{R}^n = \mathbb{R}^{n \times 1}$ and $\mathbb{R}_+$ denotes the set of all nonnegative reals. The identity matrix of order $n$ is denoted by~$I_n$ (or denoted by $I$ if the order is clear from the context). $\textbf{1}_n$ (often simply $\textbf{1}$ when its dimension is clear from the context) is the column $n$-vector of all ones. The zero matrix of order $m$ is denoted by $O_m$. $|\cdot|$ denotes the absolute value for real scalar. $A^\top$ denotes the transpose of matrix $A$. $\langle x, y \rangle = x^\top y = \sum_{i=1}^n x_i y_i$ denotes the inner product of two vectors in $\mathbb{R}^n$, $ \|x\| = \sqrt{\langle x, x \rangle}$ denotes the Euclidean norm of vector $x$ and $\|A\| = \max \big\{ \|Ax\| : x \in \mathbb{R}^n, \|x\| = 1 \big\} $ denotes the spectral norm of matrix $A$. The function $\min\{a, b\}$ denotes the minimum of real numbers $a$ and $b$. The symbol \( \otimes \) denotes the Kronecker product. For a vector $x \in \mathbb{R}^n$, its $i$-th component is denoted by $x_{(i)}$ ($i = 1, 2, \ldots, n$). The componentwise absolute value of $x\in \mathbb{R}^n$ is given by $|x| = (|x_{(1)}|, |x_{(2)}|, \ldots, |x_{(n)}|)^\top$. $\sigma_{\min}(A)$ represents the smallest singular value of matrix $A$. Additionally, $\text{tridiag}(a, b, c)$ denotes a tridiagonal matrix with subdiagonal, main diagonal, and superdiagonal entries being $a$, $b$, and $c$, respectively. $\text{Tridiag}(A, B, C)$ represents a block tridiagonal matrix with subdiagonal, main diagonal, and superdiagonal blocks being $A$, $B$, and $C$, respectively. $\text{Pentadiag}(E, F, G, H, K)$ denotes a block pentadiagonal matrix with second subdiagonal, first subdiagonal, main diagonal, first superdiagonal, and second superdiagonal blocks being $E$, $F$, $G$, $H$, and $K$, respectively. For a matrix $A\in \mathbb{R}^{n\times n}$, $D_A$ denotes its diagonal part. For a series $\{x_i\}_{i=1}^n$, $\prod\limits_{i=1}^n x_i = x_1\cdot x_2\cdot\ldots\cdot x_n$ and $\sum\limits_{i=1}^n x_i = x_1 + x_2 + \ldots+ x_n$. ${\rm rand}(n,1)$ returns an $n$-by-$1$ column vector of uniformly distributed random numbers. ${\rm mod}(x,y)$ returns the remainder when $x$ is divided by $y$.
	
\section{Preliminaries}\label{sec:pre}
In this section, we review some results of dynamical systems which lay the foundation of our later development.
	
Consider the following autonomous differential equation
\begin{equation}\label{eq:de}
		\frac{\mathrm{d}x}{\mathrm{d}t}=f(x),
\end{equation}
where $f: \mathbb{R}^n \rightarrow \mathbb{R}^n$ is a vector-valued function. A solution of \eqref{eq:de} with $x(0) = x_0$ is denoted by $x(t; x_0)$.

\begin{definition}
Let \(f:\mathbb{R}^n\rightarrow\mathbb{R}^n\).  If there exists a constant $L > 0$ such that
		\begin{equation*}
			\|f(x)-f(y)\|\leq L\|x - y\|, \quad \forall x,y\in \mathbb{R}^n,
		\end{equation*}
then \(f\) is said to be Lipschitz continuous in $\mathbb{R}^n$ and \(L \) is called the Lipschitz constant.
\end{definition}
	
\begin{definition}[see, e.g. {\cite[p. 3]{kha1996}}]
If there exists \(x_*\in\mathbb{R}^n\) such that \(f(x_*) = 0\), then \(x_*\) is called an equilibrium point of \eqref{eq:de}.
\end{definition}
	
\begin{lemma}[{\cite[Lemma 1]{poly2011}}]\label{lem:ft}
Let \(x_*\in\mathbb{R}^n\) be an equilibrium point of \eqref{eq:de}. If there exists a continuous radially unbounded function \(V:\mathbb{R}^n\rightarrow\mathbb{R}_+\) such that
\begin{enumerate}
			\item [(i)] \(V(x) = 0 \Rightarrow x = x_*\);
			\item [(ii)] There exist \(\alpha>0\), \(\beta>0\), \(0<\kappa_1 < 1\) and  \(\kappa_2>1\) such that for any solution \(x(t;x_0)\) of~\eqref{eq:de} we have
			\begin{equation*}
				\frac{\mathrm{d}V(x(t;x_0))}{\mathrm{d}t}\leq-\alpha V(x(t;x_0))^{\kappa_1}-\beta V(x(t;x_0))^{\kappa_2}.
			\end{equation*}
\end{enumerate}
Then the equilibrium point $x_*$ of \eqref{eq:de} is globally fixed-time stable and
		\begin{equation*}T(x_0)\leq T_{\max}=\frac{1}{\alpha(1 - \kappa_1)}+\frac{1}{\beta(\kappa_2 - 1)}, \quad \forall x_0  \in\mathbb{R}^n,
		\end{equation*}
where $T: \mathbb{R}^n \rightarrow \mathbb{R}_+$ is the settling-time function.
\end{lemma}
	
\begin{lemma}[{\cite[Proposition~1]{gbgm2022}}]\label{lem:solution}
Let \(f:\mathbb{R}^n\rightarrow\mathbb{R}^n\) be a locally Lipschitz continuous vector-valued function such that
		\begin{equation*}
			f(x_*) = 0\quad\text{and}\quad\langle x - x_*, f(x)\rangle >0
		\end{equation*}
		for any \(x\in\mathbb{R}^n\setminus\{x_*\}\), where \(x_*\in\mathbb{R}^n\). Consider the following autonomous differential equation
		\begin{equation}\label{eq:ad}
			\frac{\mathrm{d}x}{\mathrm{d}t}=-\rho(x)f(x),
		\end{equation}
	where
		\begin{equation*}
			\rho(x):=\begin{cases}
				\frac{\rho_1}{\|f(x)\|^{1 - \lambda_1}}+\frac{\rho_2}{\|f(x)\|^{1 - \lambda_2}},&\text{if}~f(x)\neq0,\\
				0,&\text{if}~f(x) = 0
			\end{cases}
\end{equation*}
with \(\rho_1,\rho_2>0\), \(\lambda_1\in(0,1)\) and \(\lambda_2 > 1\). Then the right-hand side of the equation \eqref{eq:ad} is continuous for all \(x\in\mathbb{R}^n\). Moreover, for any given initial point \(x(0) = x_0\), a solution of \eqref{eq:ad} exists and is uniquely determined for   $t\ge 0$.
\end{lemma}
	
\section{The fixed-time stable dynamical model for EVLCP}\label{sec:dm}
	In this section, we will develop a fixed-time stable and inverse-free dynamical model for EVLCP~\eqref{eq:evlcp}. To this end, we first study a kind of new generalized absolute value equations~(NGAVE).
	
	For any given $B_0,B_i^j\in \mathbb{R}^{n \times n}$ and $b_0,b_i^j\in \mathbb{R}^{n}$ with $i=1,2,\ldots,s$ and $j=1,2,\ldots,2^{i-1}$, define $\phi_0(x)$ and $\phi_i[\cdot](x)$ as follows:
	{\small\begin{equation}\label{eq:phis}
			\begin{aligned}
				\phi_0(x) &= B_0 x + b_0, \\
				\phi_1[B_1^1,b_1^1](x) &= |B_1^1 x + b_1^1|, \\
				\phi_2[B_2^1,B_2^2,b_2^1,b_2^2](x) &= |B_2^1 x + b_2^1 + \phi_1[B_2^2,b_2^2](x)|, \\
				\phi_3[B_3^1,\ldots,B_3^4,b_3^1,\ldots,b_3^4](x) &= |B_3^1 x + b_3^1 + \phi_1[B_3^2,b_3^2](x) + \phi_2[B_3^3,B_3^4,b_3^3,b_3^4](x) |, \\
				\phi_4[B_4^1,\ldots,B_4^8,b_4^1,\ldots,b_4^8](x) &= |  B_4^1 x + b_4^1 + \phi_1[B_4^2,b_4^2](x) + \phi_2[B_4^3,B_4^4,b_4^3,b_4^4](x)
				\\&\qquad + \phi_3[B_4^5,\ldots,B_4^8,b_4^5,\ldots,b_4^8](x) | ,\\
				&\ \, \vdots   \\
				\phi_s[B_s^1,\ldots,B_s^{2^{s -1}}, b_s^1,\ldots,b_s^{2^{s -1}}](x)
				&= |B_s^1x + b_s^1 + \phi_1[B_s^2,b_s^2](x) + \phi_2[B_s^3,B_s^3,b_s^3,b_s^4](x)
				\\&\qquad+ \phi_3[B_s^5,\ldots,B_s^8,b_s^5,\ldots,b_s^8](x)
				+\ldots\\
				&\qquad +\phi_{s-1}[B_s^{2^{s-2} +1},\ldots,
				B_s^{2^{s-1}}, b_s^{2^{s-2} +1},\ldots,b_s^{2^{s-1}}](x) |.
			\end{aligned}
	\end{equation}}
	Then for any $s \ge 1$, we consider the following NGAVE
	\begin{equation}\label{eq:nngave}
		\varphi_s(x)=\phi_0(x) + \sum_{i=1}^{s} \phi_i[B_i^1, \dots, B_i^{2^{i-1}}, b_i^1, \dots, b_i^{2^{i-1}}](x) = 0.
	\end{equation}
	In order to develop a fixed-time stable dynamical model for solving NGAVE~\eqref{eq:nngave}, we first explore some  properties of  NGAVE~\eqref{eq:nngave} in the following.
	
	\begin{lemma}\label{lem:key}
		Let \(B_i^j\in \mathbb{R}^{n \times n}\) and \(b_i^j\in \mathbb{R}^{n}\) with \(i=1,2,\ldots,s\) and \(j=1,2,\ldots,2^{i-1}\). Assume that  \(\phi_i[\cdot] (i=1,2,\ldots,s) \) are defined as in \eqref{eq:phis}.  Then, for $i=1,2,\ldots,s$, we have
		\begin{equation*}
			\left\|\phi_i[B_i^1,\ldots,B_i^{2^{i-1}},b_i^1,\ldots,b_i^{2^{i-1}}](x) - \phi_i[B_i^1,\ldots,B_i^{2^{i-1}},b_i^1,\ldots,b_i^{2^{i-1}}](y)\right\| \leq  \left(\sum_{j=1}^{2^{i-1}} \|B_i^j\|\right)\|x - y\|.
		\end{equation*}
	\end{lemma}	
	\begin{proof}
		For $i = 1$, it follows from $\||x| - |y|\|\le \|x-y\|$ that
		\begin{equation*}
			\left\|\phi_1[B_1^1,b_1^1](x) - \phi_1[B_1^1,b_1^1](y)\right\| \leq \|B_1^1\| \|x - y\|.
		\end{equation*}
		Assume that for \(i=2,3,\ldots ,s-1\) we have
		\begin{equation*}
			\left\|\phi_i[B_i^1,\ldots,B_i^{2^{i-1}}, b_i^1,\ldots,b_i^{2^{i -1}}](x) - \phi_i[B_i^1,\ldots,B_i^{2^{i-1}}, b_i^1,\ldots,b_i^{2^{i -1}}](y)\right\| \leq \left(\sum_{j=1}^{2^{i-1}} \|B_i^j\|\right) \|x - y\|.
		\end{equation*}
		Then, for \(i=s\), we have
		\begin{align*}
			&\left\|\phi_s[B_s^1,\ldots,B_s^{2^{s -1}}, b_s^1,\ldots,b_s^{2^{s -1}}](x) - \phi_s[B_s^1,\ldots,B_s^{2^{s -1}}, b_s^1,\ldots,b_s^{2^{s -1}}](y)\right\|\\
			&\qquad = \Big\|\big|B_s^1x + b_s^1 + \phi_1[B_s^2,b_s^2](x) + \phi_2[B_s^3,B_s^4,b_s^3,b_s^4](x)
			+ \phi_3[B_s^5,\ldots,B_s^8,b_s^5,\ldots,b_s^8](x)\\
			&\qquad\qquad +\ldots+\phi_{s-1}[B_s^{2^{s-2} +1},\ldots,
			B_s^{2^{s-1}}, b_s^{2^{s-2} +1},\ldots,b_s^{2^{s-1}}](x) \big|\\
			&\qquad \qquad-\big|B_s^1y + b_s^1 + \phi_1[B_s^2,b_s^2](y) + \phi_2[B_s^3,B_s^4,b_s^3,b_s^4](y)
			+ \phi_3[B_s^5,\ldots,B_s^8,b_s^5,\ldots,b_s^8](y)\\
			&\qquad \qquad +\ldots+\phi_{s-1}[B_s^{2^{s-2} +1},\ldots,
			B_s^{2^{s-1}}, b_s^{2^{s-2} +1},\ldots,b_s^{2^{s-1}}](y) \big|\Big\|\\
			&\qquad \leq \|B_s^1\|\|x-y\|+\Big\|\phi_1[B_s^2,b_s^2](x) + \phi_2[B_s^3,B_s^4,b_s^3,b_s^4](x)
			+ \phi_3[B_s^5,\ldots,B_s^8,b_s^5,\ldots,b_s^8](x)\\
			&\qquad \qquad +\ldots+\phi_{s-1}[B_s^{2^{s-2} +1},\ldots,
			B_s^{2^{s-1}}, b_s^{2^{s-2} +1},\ldots,b_s^{2^{s-1}}](x) \\
			&\qquad \qquad-\big[\phi_1[B_s^2,b_s^2](y) + \phi_2[B_s^3,B_s^4,b_s^3,b_s^4](y)
			+ \phi_3[B_s^5,\ldots,B_s^8,b_s^5,\ldots,b_s^8](y)\\
			&\qquad \qquad +\ldots+\phi_{s-1}[B_s^{2^{s-2} +1},\ldots,
			B_s^{2^{s-1}}, b_s^{2^{s-2} +1},\ldots,b_s^{2^{s-1}}](y) \big]\Big\|\\
			&\qquad \leq \|B_s^1\|\|x-y\|+\|B_s^2\|\|x-y\|+ \Big\| \phi_2[B_s^3,B_s^4,b_s^3,b_s^4](x)
			+ \phi_3[B_s^5,\ldots,B_s^8,b_s^5,\ldots,b_s^8](x)\\
			&\qquad \qquad +\ldots+\phi_{s-1}[B_s^{2^{s-2} +1},\ldots,
			B_s^{2^{s-1}}, b_s^{2^{s-2} +1},\ldots,b_s^{2^{s-1}}](x) \\
			&\qquad \qquad-\big[\phi_2[B_s^3,B_s^4,b_s^3,b_s^4](y)
			+ \phi_3[B_s^5,\ldots,B_s^8,b_s^5,\ldots,b_s^8](y)\\
			&\qquad \qquad +\ldots+\phi_{s-1}[B_s^{2^{s-2} +1},\ldots,
			B_s^{2^{s-1}}, b_s^{2^{s-2} +1},\ldots,b_s^{2^{s-1}}](y)\big]\Big\|\\
			&\qquad \leq \left(\sum_{i=1}^{2^{s-2}}\|B_s^i\|\right)\|x-y\|+\Big\| \phi_{s-1}[B_s^{2^{s-2} +1},\ldots,B_s^{2^{s-1}}, b_s^{2^{s-2} +1},\ldots,b_s^{2^{s-1}}](x)\\
			&\qquad \qquad-\phi_{s-1}[B_s^{2^{s-2} +1},\ldots,
			B_s^{2^{s-1}}, b_s^{2^{s-2} +1},\ldots,b_s^{2^{s-1}}](y) \Big\|\\
			&\qquad \leq \left(\sum_{i=1}^{2^{s-2}}\|B_s^i\| + \sum_{i=2^{s-2} +1}^{2^{s-1}}\|B_s^i\|\right)\|x-y\|
			\\
			&\qquad= \left(\sum_{i=1}^{2^{s-1}} \|B_s^i\|\right)\|x - y\|.
		\end{align*}
		The proof is completed by induction.
	\end{proof}

	By Lemma ~\ref{lem:key}, for $\forall x,y\in \mathbb{R}^n$, we obtain that
\begin{align}\nonumber
				&\|\varphi_s(x)-\varphi_s(y)\|\\\nonumber
				&\qquad =\left\| \phi_0(x) + \sum_{i=1}^{s} \phi_i[B_i^1, \dots, B_i^{2^{i-1}}, b_i^1, \dots, b_i^{2^{i-1}}](x)\right.\\\nonumber
&\qquad\qquad \left.-\left[\phi_0(y) + \sum_{i=1}^{s} \phi_i[B_i^1, \dots, B_i^{2^{i-1}}, b_i^1, \dots, b_i^{2^{i-1}}](y)\right]\right\|\\	\label{eq:ngavee}
&\qquad \leq\left(\|B_{0}\|+\sum_{i=1}^s \sum_{j=1}^{2^{i-1}} \|B_i^j\|\right)\|x - y\|.
\end{align}
	
\begin{theorem}\label{thm:uunique}
		Assume that \(B_0\) and \(B_i^j\in \mathbb{R}^{n \times n}\) with \(i=1,2,\ldots,s\) and \(j=1,2,\ldots,2^{i-1}\).  If
		\begin{equation}\label{eq:cc}
			\sigma_{\min}(B_0) > \sum_{i=1}^s \sum_{j=1}^{2^{i-1}} \|B_i^j\|,	
		\end{equation}
then NGAVE~\eqref{eq:nngave} has a unique solution for any \(b_0,b_i^j\in \mathbb{R}^{n}.\)
\end{theorem}
	
\begin{proof}
If \eqref{eq:cc} holds, $B_0$ is nonsingular. Then we can define the mapping \(M: \mathbb{R}^{n} \to \mathbb{R}^{n}\) as
		\begin{equation*}
			M(x) = -B_0^{-1} \left( \sum_{i=1}^{s} \phi_i[B_i^1, \dots, B_i^{2^{i-1}}, b_i^1, \dots, b_i^{2^{i-1}}](x) + b_0 \right).
		\end{equation*}
For \(\forall x, y \in \mathbb{R}^{n}\), we have
\begin{align*}
\|M(x) - M(y)\|
& \leq \|B_0^{-1}\| \left\| \sum_{i=1}^{s} \phi_i[B_i^1, \dots, B_i^{2^{i-1}}, b_i^1, \dots, b_i^{2^{i-1}}](x) \right.\\
&\qquad\left.- \sum_{i=1}^{s} \phi_i[B_i^1, \dots, B_i^{2^{i-1}}, b_i^1, \dots, b_i^{2^{i-1}}](y)\right\| \\
				& \leq \|B_0^{-1}\| \left( \sum_{i=1}^s\sum_{j=1}^{2^{i-1}} \|B_i^j\| \right)\|x - y\|,
\end{align*}
which together with \eqref{eq:cc} and $\|B_0^{-1}\| = 1/\sigma_{\min}(B_0)$  implies that
		\(M\) is a contraction mapping in $\mathbb{R}^n$. By the Banach fixed-point theorem, \(M\) has a unique fixed point in $\mathbb{R}^n$, which implies that NGAVE~\eqref{eq:nngave} has a unique solution for any \(b_0,b_i^j\in \mathbb{R}^{n}\).
	\end{proof}

\begin{theorem}\label{thm:e1}
		Let \(B_0,B_i^j\in \mathbb{R}^{n \times n}\) and \(b_0,b_i^j\in \mathbb{R}^{n}\) with \(i=1,2,\ldots,s\) and \(j=1,2,\ldots,2^{i-1}\). If~\eqref{eq:cc} holds, then for $\forall x\in \mathbb{R}^n$ we have
		\begin{equation}\label{eq:qf}
			(x - x_*)^\top B_{0}^\top\varphi_s(x)\geq \frac{1}{2} \|\varphi_s(x)\|^2,
		\end{equation}
		where \(x_*\) is the unique solution to NGAVE~\eqref{eq:nngave}.
	\end{theorem}
	
\begin{proof}
According to Theorem~\ref{thm:uunique}, NGAVE~\eqref{eq:nngave} has a unique solution $x_*$ whenever \eqref{eq:cc} holds. Then for \(\forall x \in \mathbb{R}^{n}\), we have
		{\footnotesize\begin{align*}
				(x -& x_*)^\top B_{0}^\top \varphi_s(x) - \frac{1}{2} \|\varphi_s(x)\|^2 \\
				&= (x - x_*)^\top B_{0}^\top \left(\varphi_s(x) - \varphi_s(x_*) \right)
				- \frac{1}{2} \|\varphi_s(x) - \varphi_s(x_*)\|^2 \\
				&= (x - x_*)^\top B_{0}^\top B_0(x - x_*)\\
				&\qquad + (x - x_*)^\top B_{0}^\top\left(\sum_{i=1}^{s} \phi_i[B_i^1, \dots, B_i^{2^{i-1}}, b_i^1, \dots, b_i^{2^{i-1}}](x) - \sum_{i=1}^{s} \phi_i[B_i^1, \dots, B_i^{2^{i-1}}, b_i^1, \dots, b_i^{2^{i-1}}](x_*)\right)\\
				&\qquad- \frac{1}{2} \left\|B_0(x - x_*) + \sum_{i=1}^{s} \phi_i[B_i^1, \dots, B_i^{2^{i-1}}, b_i^1, \dots, b_i^{2^{i-1}}](x) - \sum_{i=1}^{s} \phi_i[B_i^1, \dots, B_i^{2^{i-1}}, b_i^1, \dots, b_i^{2^{i-1}}](x_*)\right\|^2 \\
				&= \| B_0 (x - x_*)\|^2  - \frac{1}{2} \| B_0  \left(x - x_*\right) \|^2 \\
				&\qquad + (x - x_*)^\top B_0^\top \left( \sum_{i=1}^{s} \phi_i[B_i^1, \dots, B_i^{2^{i-1}}, b_i^1, \dots, b_i^{2^{i-1}}](x) - \sum_{i=1}^{s} \phi_i[B_i^1, \dots, B_i^{2^{i-1}}, b_i^1, \dots, b_i^{2^{i-1}}](x_*) \right)  \\
				&\qquad - \frac{1}{2} \left\| \sum_{i=1}^{s} \phi_i[B_i^1, \dots, B_i^{2^{i-1}}, b_i^1, \dots, b_i^{2^{i-1}}](x) - \sum_{i=1}^{s} \phi_i[B_i^1, \dots, B_i^{2^{i-1}}, b_i^1, \dots, b_i^{2^{i-1}}](x_*) \right\|^2 \\
				&\qquad - (x - x_*)^\top B_0^\top \left( \sum_{i=1}^{s} \phi_i[B_i^1, \dots, B_i^{2^{i-1}}, b_i^1, \dots, b_i^{2^{i-1}}](x) - \sum_{i=1}^{s} \phi_i[B_i^1, \dots, B_i^{2^{i-1}}, b_i^1, \dots, b_i^{2^{i-1}}](x_*) \right)\\
				&= \frac{1}{2} \| B_0 (x - x_*)\|^2 - \frac{1}{2} \left\| \sum_{i=1}^{s} \phi_i[B_i^1, \dots, B_i^{2^{i-1}}, b_i^1, \dots, b_i^{2^{i-1}}](x) - \sum_{i=1}^{s} \phi_i[B_i^1, \dots, B_i^{2^{i-1}}, b_i^1, \dots, b_i^{2^{i-1}}](x_*) \right\|^2 \\
				&\geq \frac{1}{2} \sigma_{\text{min}}^2 (B_0) \| x - x_* \|^2 \\
				&\qquad - \frac{1}{2} \left\| \sum_{i=1}^{s} \phi_i[B_i^1, \dots, B_i^{2^{i-1}}, b_i^1, \dots, b_i^{2^{i-1}}](x) - \sum_{i=1}^{s} \phi_i[B_i^1, \dots, B_i^{2^{i-1}}, b_i^1, \dots, b_i^{2^{i-1}}](x_*)\right\|^2 \\
				&\geq \frac{1}{2} \sigma_{\text{min}}^2 (B_0) \| x - x_* \|^2- \frac{1}{2} \left( \left( \sum_{i=1}^s \sum_{j=1}^{2^{i-1}} \| B_i^j \| \right) \| x - x_* \|\right)^2 \\
				&= \frac{1}{2} \left(\sigma_{\min}^2(B_0) - \left( \sum_{i=1}^s \sum_{j=1}^{2^{i-1}} \|B_i^j\| \right)^2\right) \|x - x_*\|^2\\
				&= \frac{1}{2} \left(\sigma_{\min}(B_0) - \sum_{i=1}^s \sum_{j=1}^{2^{i-1}} \|B_i^j\| \right)\left(\sigma_{\min}(B_0) + \sum_{i=1}^s \sum_{j=1}^{2^{i-1}} \|B_i^j\| \right)\|x - x_*\|^2\\
				&\geq 0.
		\end{align*}}
		The final inequality holds due to \eqref{eq:cc} and \(\|x - x_*\| \geq 0\).
	\end{proof}
	
	\begin{theorem}
		Let \(B_0,B_{ij}\in \mathbb{R}^{n \times n}\) and \(b_0,b_{ij}\in \mathbb{R}^{n}\) with \(i=1,2,\ldots,s\) and \(j=1,2,\ldots,2^{i-1}\). If~\eqref{eq:cc} holds, then for any \(x \in \mathbb{R}^{n}\)  we have
		\begin{equation}\label{eq:dvvt}
			\frac{\| \varphi_s(x) \|}{\|B_0\|+ \sum_{i=1}^s \sum_{j=1}^{2^{i-1}} \|B_i^j\|}  \leq \|x - x_*\|  \leq\frac{\| \varphi_s(x)\|}{\sigma_{\min}(B_0) - \sum_{i=1}^s \sum_{j=1}^{2^{i-1}} \|B_i^j\|} ,
		\end{equation}
		where \(x_*\) is the unique solution to NGAVE~\eqref{eq:nngave} .
	\end{theorem}
	
	\begin{proof}
		Once again, it follows from Theorem~\ref{thm:uunique} and \eqref{eq:cc} that NGAVE~\eqref{eq:nngave} has a unique solution $x_*$. Then for any \(x \in \mathbb{R}^{n}\)  we have
		\begin{align*}
			\| \varphi_s(x)\|&=\|\varphi_s(x) - \varphi_s(x_*)\|\\
			&= \left\| B_0(x - x_*) +  \sum_{i=1}^{s} \phi_i[B_i^1, \dots, B_i^{2^{i-1}}, b_i^1, \dots, b_i^{2^{i-1}}](x) \right.\\
			&\qquad \left.- \sum_{i=1}^{s} \phi_i[B_i^1, \dots, B_i^{2^{i-1}}, b_i^1, \dots, b_i^{2^{i-1}}](x_*) \right\| \\
			&\leq \|B_0\| \|x - x_*\| + \sum_{i=1}^{s}\left(\left( \sum_{j=1}^{2^{i-1}} \|B_i^j\| \right)\|x - x_*\| \right)\\
			&= \left( \|B_0\| + \sum_{i=1}^s \sum_{j=1}^{2^{i-1}} \|B_i^j\|\right)\|x - x_*\|,
		\end{align*}
		which together with $\|B_0\|+ \sum_{i=1}^s \sum_{j=1}^{2^{i-1}} \|B_i^j\| > 0$ implies
		\begin{align}\label{eq:dtt}
			\frac{1}{\|B_0\|+ \sum_{i=1}^s \sum_{j=1}^{2^{i-1}} \|B_i^j\|} \| \varphi_s(x) \| \leq \|x - x_*\|.
		\end{align}
		In addition,
		\begin{align*}
			\|\varphi_s(x)\|&=\|\varphi_s(x) - \varphi_s(x_*)\|\\
			&\geq\|B_0(x - x_*)\| - \left\| \sum_{i=1}^{s} \phi_i[B_i^1, \dots, B_i^{2^{i-1}}, b_i^1, \dots, b_i^{2^{i-1}}](x)\right.\\
			&\qquad \left.- \sum_{i=1}^{s} \phi_i[B_i^1, \dots, B_i^{2^{i-1}}, b_i^1, \dots, b_i^{2^{i-1}}](x_*)\right \| \\
			&\geq \sigma_{\min}(B_0)\|x - x_*\| - \sum_{i=1}^{s}\left(\left( \sum_{j=1}^{2^{i-1}} \|B_i^j\| \right)\|x - x_*\| \right) \\
			&\geq
			\left(\sigma_{\min}(B_0) - \sum_{i=1}^s \sum_{j=1}^{2^{i-1}} \|B_i^j\| \right)\|x - x_*\|,
\end{align*}
		which together with \eqref{eq:cc} implies
		\begin{align}\label{eq:b}
			\|x - x_*\|  \leq\frac{1}{\sigma_{\min}(B_0) - \sum_{i=1}^s \sum_{j=1}^{2^{i-1}} \|B_i^j\|} \| \varphi_s(x)\|.
		\end{align}
		Combine \eqref{eq:dtt} and \eqref{eq:b}, we obtain \eqref{eq:dvvt}.
	\end{proof}

	Now we are in the position to develop the following dynamical model for solving NGAVE~\eqref{eq:nngave}:
	\begin{equation}\label{eq:mr}
		\frac{{\rm d}x}{{\rm d}t} = -\rho(x) g(\gamma, x),
	\end{equation}
	where
	\begin{equation}\label{eq:mrd}
		\rho(x) =
		\begin{cases}
			\frac{\rho_1}{\|g(\gamma,x)\|^{1-\lambda_1}} + \frac{\rho_2}{\|g(\gamma,x)\|^{1-\lambda_2}}, & \text{if } x \notin \text{SOL}(g), \\
			0, & \text{if } x \in \text{SOL}(g),
		\end{cases}
	\end{equation}
	$g(\gamma, x) = \gamma B_0^\top \varphi_s(x)$ with $\varphi_s(x)$ being defined as in \eqref{eq:nngave}, $\rho_1, \rho_2, \gamma > 0$, $\lambda_1 \in (0, 1)$, and  $\lambda_2 \in (1, +\infty)$. Here,
	\begin{equation*}
		\text{SOL}(g) := \{x \in \mathbb{R}^n : B_0^\top \varphi_s(x) = 0\}.
	\end{equation*}
	
	Based on the properties of NGAVE~\eqref{eq:nngave} mentioned above, we can explore some properties of the dynamical model \eqref{eq:mr}--\eqref{eq:mrd}. Especially, we will prove that it is fixed-time stable.
	
	We first investigate the equivalence between the equilibrium point of the dynamical  model~\eqref{eq:mr}--\eqref{eq:mrd} and the solution  of NGAVE~\eqref{eq:nngave}. Similar to the proof of \cite[Theorem 3.1]{lyyh2023}, we have the following Theorem~\ref{thm:lipq}
	
	\begin{theorem}\label{thm:lipq}
		Assume that \(B_0,B_i^j\in \mathbb{R}^{n \times n}\) and \(b_0,b_i^j\in \mathbb{R}^{n}\) with \(i=1,2,\ldots,s\) and \(j=1,2,\ldots,2^{i-1}\). If \eqref{eq:cc} holds, then the dynamical model~\eqref{eq:mr}--\eqref{eq:mrd} has a unique equilibrium point
		$x_*$, which is equivalent to the unique solution to  NGAVE~\eqref{eq:nngave}.
	\end{theorem}
	
	\begin{theorem}\label{thm:e2}
		The function $g(\gamma, x)$ in \eqref{eq:mr} is Lipschitz continuous on $\mathbb{R}^n$.
	\end{theorem}
	
	\begin{proof}
		For $\forall x, y \in \mathbb{R}^n$, it follows from \eqref{eq:ngavee} that
		\begin{align*}
			\|g(\gamma, x) - g(\gamma, y)\|
			&= \gamma \|B_0^T(\varphi_s(x) - \varphi_s(y))\|\\
			&\le \gamma  \|B_0^T\| \left(\|B_{0}\|+\sum_{i=1}^s \sum_{j=1}^{2^{i-1}} \|B_i^j\|\right)\|x - y\|.
		\end{align*}
		Thus, the function $g(\gamma, x)$ is Lipschitz continuous on $\mathbb{R}^n$ with the Lipschitz constant being  $\gamma  \|B_0^T\| \left(\|B_{0}\|+\sum_{i=1}^s \sum_{j=1}^{2^{i-1}} \|B_i^j\|\right)$.
	\end{proof}
	Combine Theorem~\ref{thm:e1}, Lemma~\ref{lem:solution} and Theorem~\ref{thm:e2}, we obtain the following theorem.
	
	\begin{theorem}\label{thm:solutionn}
		Suppose that \(B_0,B_i^j\in \mathbb{R}^{n \times n}\) and \(b_0,b_i^j\in \mathbb{R}^{n}\) with \(i=1,2,\ldots,s\) and \(j=1,2,\ldots,2^{i-1}\). If \eqref{eq:cc} holds, then, for any given initial point $x(0) = x_0\in \mathbb{R}^n$, the  dynamical  model~\eqref{eq:mr}--\eqref{eq:mrd} has a unique solution $x(t; x_0)$  with $t \in [0, +\infty)$.
	\end{theorem}

	Next, we analyze the fixed-time stability of the equilibrium point for the dynamical model~\eqref{eq:mr}--\eqref{eq:mrd}.
	
	\begin{theorem}
		Suppose that \(B_0,B_i^j\in \mathbb{R}^{n \times n}\) and \(b_0,b_i^j\in \mathbb{R}^{n}\) with \(i=1,2,\ldots,s\) and \(j=1,2,\ldots,2^{i-1}\). If \eqref{eq:cc} holds, then the unique equilibrium point $x_*$ of the dynamical model~\eqref{eq:mr}--\eqref{eq:mrd} is globally fixed-time stable. Concretely, for any given initial point \( x(0) = x_0\in \mathbb{R}^n\), we have
		\begin{equation}\label{eq:tmax}
			T(x_0) \leq T_{\max} = \frac{1}{c_1(1 - \kappa_1)} + \frac{1}{c_2(\kappa_2 - 1)},
		\end{equation}
		where  $c_1$, $c_2$, $\kappa_1$, and $\kappa_2$ are defined as in ~\eqref{eq:bi}
	\end{theorem}
	\begin{proof}
By Theorem~\ref{thm:lipq} and Theorem~\ref{thm:solutionn}, we know that the dynamical model~\eqref{eq:mr}--\eqref{eq:mrd} has a unique equilibrium point $x_*$ and  for any $x(0) = x_0 \in \mathbb{R}^n$ it has a unique solution $x = x(t; x_0)$ with $t\ge 0$.
		
		Let
		\begin{equation}\label{eq:fl}
			V(x) = \frac{1}{2} \|x - x_*\|^2.
		\end{equation}
		Then we have $V(x) \to \infty$ as $\|x - x_*\| \to \infty$ and $V(x) = 0$ if and only if $x = x_*$.
		
		For any given $x_0 \in \mathbb{R}^n \setminus \{x_*\}$, before the solution $x = x(t; x_0)$ arriving $x_*$, it follows from~\eqref{eq:mr}, \eqref{eq:mrd} and \eqref{eq:fl} that
\begin{align}\nonumber
		\frac{{\rm d} V(x)}{{\rm d}t} &= (x - x_*)^\top \frac{{\rm d}x}{{\rm d}t} \\\nonumber
		&= -\gamma\rho(x)(x - x_*)^\top B_0^\top \varphi_s(x) \\\label{ie:d1}
		&= -\gamma \rho_1 \frac{( x - x_*)^\top B_0^\top  \varphi_s(x)}{\|\gamma B_0^\top \varphi_s(x)\|^{1-\lambda_1}} - \gamma \rho_2 \frac{( x - x_*)^\top B_0^\top  \varphi_s(x) }{\|\gamma B_0^\top \varphi_s(x)\|^{1-\lambda_2}}.
	\end{align}
According to \eqref{eq:cc}, \eqref{eq:qf} and the second inequality of \eqref{eq:dvvt}, we have
\begin{align}\nonumber
		\gamma \rho_1 \frac{( x - x_*)^\top B_0^\top \varphi_s(x)}{\|\gamma B_0^\top  \varphi_s(x)\|^{1-\lambda_1}} &\geq \frac{1}{2}\gamma \rho_1 \frac{\| \varphi_s(x)\|^2 }{\|\gamma B_0^\top  \varphi_s(x)\|^{1-\lambda_1}} \\\nonumber
		&\geq \frac{1}{2}\gamma \rho_1 \frac{\| \varphi_s(x)\|^2 }{\|\gamma B_0^\top \|^{1-\lambda_1}\|\varphi_s(x)\|^{1-\lambda_1}} \\\nonumber
	    &= \frac{1}{2}\gamma \rho_1 \frac{1}{\|\gamma B_0^\top \|^{1-\lambda_1}}\|\varphi_s(x)\|^{1+\lambda_1} \\\label{ie:d2}
	    &\geq \frac{\rho_1\gamma^{\lambda_1} }{2\|B_0\|^{1-\lambda_1}} \left(\sigma_{\min}(B_0) - \sum_{i=1}^s \sum_{j=1}^{2^{i-1}} \|B_i^j\| \right)^{1+\lambda_1}\|x-x_*\|^{1+\lambda_1}
\end{align}
and
\begin{align}\nonumber
	\gamma \rho_2 \frac{( x - x_*)^\top B_0^\top \varphi_s(x)}{\|\gamma B_0^\top  \varphi_s(x)\|^{1-\lambda_2}} &\geq \frac{1}{2}\gamma \rho_2 \frac{\| \varphi_s(x)\|^2 }{\|\gamma B_0^\top  \varphi_s(x)\|^{1-\lambda_2}} \\\nonumber
	&= \frac{1}{2}\gamma \rho_2\|\varphi_s(x)\|^{2} \|\gamma B_0^\top \varphi_s(x)\|^{\lambda_2-1} \\\nonumber
	&\geq \frac{1}{2}\gamma \rho_2\|\varphi_s(x)\|^{2} \left(\frac{\gamma\|\varphi_s(x)\|^{2}}{2\|x-x_*\|}\right)^{\lambda_2-1} \\\nonumber
	&= \frac{1}{2}\gamma \rho_2\left(\frac{\gamma}{2}\right)^{\lambda_2-1}\frac{\|\varphi_s(x)\|^{2\lambda_2}}{\|x-x_*\|^{\lambda_2-1}} \\\label{ie:d3}
	&\geq \frac{\rho_2 \gamma^{\lambda_2}}{2^{\lambda_2}} \left(\sigma_{\min}(B_0) - \sum_{i=1}^s \sum_{j=1}^{2^{i-1}} \|B_i^j\| \right)^{2\lambda_2}\|x-x_*\|^{1+\lambda_2}.
\end{align}

Combine \eqref{ie:d1}, \eqref{ie:d2} and \eqref{ie:d3}, we get
\begin{align*}
	\frac{{\rm d}V(x)}{{\rm d} t} &\leq -\frac{\rho_1\gamma^{\lambda_1} }{2\|B_0\|^{1-\lambda_1}} \left(\sigma_{\min}(B_0) - \sum_{i=1}^s \sum_{j=1}^{2^{i-1}} \|B_i^j\| \right)^{1+\lambda_1}\|x-x_*\|^{1+\lambda_1} \\
	&\qquad - \frac{\rho_2 \gamma^{\lambda_2}}{2^{\lambda_2}} \left(\sigma_{\min}(B_0) - \sum_{i=1}^s \sum_{j=1}^{2^{i-1}} \|B_i^j\| \right)^{2\lambda_2}\|x-x_*\|^{1+\lambda_2}\\
	&= - \frac{2^{\frac{\lambda_1-1}{2}}\rho_1\gamma^{\lambda_1}}{\| B_0\|^{1-\lambda_1}} \left(\sigma_{\min}(B_0) - \sum_{i=1}^s \sum_{j=1}^{2^{i-1}} \|B_i^j\| \right)^{1+\lambda_1} \left(\frac{1}{2} \|x - x_*\|^2\right)^{\frac{\lambda_1+1}{2}} \\
	&\qquad - 2^{\frac{1-\lambda_2}{2}} \rho_2 \gamma^{\lambda_2} \left(\sigma_{\min}(B_0) - \sum_{i=1}^s \sum_{j=1}^{2^{i-1}} \|B_i^j\| \right)^{2\lambda_2} \left(\frac{1}{2} \|x - x_*\|^2\right)^{\frac{1+\lambda_2}{2}} \\
	&= -c_1 V(x)^{\kappa_1} - c_2 V(x)^{\kappa_2},
\end{align*}
where
\begin{equation}\label{eq:bi}
	\begin{aligned}
		c_1 &= \frac{2^{\frac{\lambda_1-1}{2}}\rho_1\gamma^{\lambda_1}}{\| B_0\|^{1-\lambda_1}} \left(\sigma_{\min}(B_0) - \sum_{i=1}^s \sum_{j=1}^{2^{i-1}} \|B_i^j\| \right)^{1+\lambda_1} > 0, \quad \kappa_1 = \frac{\lambda_1 + 1}{2} \in (0.5, 1), \\
		c_2 &= 2^{\frac{1-\lambda_2}{2}} \rho_2 \gamma^{\lambda_2} \left(\sigma_{\min}(B_0) - \sum_{i=1}^s \sum_{j=1}^{2^{i-1}} \|B_i^j\| \right)^{2\lambda_2} > 0, \quad \kappa_2 = \frac{\lambda_2 + 1}{2} \in (1, +\infty).
	\end{aligned}
\end{equation}
The proof is completed by Lemma~\ref{lem:ft}.
	\end{proof}
	
	\begin{remark}
		For $s\ge 2$, by $\min\{\alpha,\beta\} = \frac{\alpha + \beta - |\alpha - \beta|}{2}$ and $\min\{\alpha,\beta,\gamma\} = \min\{\alpha, \min\{\beta,\gamma\}\}$, EVLCP~\eqref{eq:evlcp} can be equivalently transformed into a special case of NGAVE~\eqref{eq:nngave} \cite{hevo2022}. For example, for $s = 2$, EVLCP~\eqref{eq:evlcp} is equivalent to
		\begin{align*}
			(2\Omega A_0 &+ A_1 + A_2)x - |(A_1 - A_2)x + q_1 - q_2| \\
			&- \left|(2\Omega A_0 - A_1 -A_2)x + |(A_1 - A_2)x + q_1 -q_2 | + 2\Omega q_0 - q_1 - q_2\right|  + 2\Omega q_0 + q_1 + q_2 = 0,
		\end{align*}
		where $\Omega \in \mathbb{R}^{n\times n}$ is a positive diagonal matrix. Then we have \( B_0 = -(2\Omega A_{0} + A_{1} + A_{2}) \), \( B_1^1 =B_2^2= A_{1} - A_{2} \), \( B_2^1 = 2\Omega A_{0} - A_{1} - A_{2} \),  $b_0 = -(2\Omega q_0 + q_1 + q_2), b_1^1 = b_2^2 = q_1 - q_2$ and $b_2^1 = 2\Omega q_0 -q_1 - q_2$. The dynamical model \eqref{eq:mr}--\eqref{eq:mrd} can be used to solve EVLCP~\eqref{eq:evlcp} with $s = 2$ and  the condition \eqref{eq:cc} becomes
\begin{equation}\label{ie:cons2}
			\sigma_{\min} (-(2\Omega A_{0} + A_{1} + A_{2}))> 2\|A_1 - A_2\| + \|2\Omega A_{0} - A_{1} - A_{2}\|.	
		\end{equation}
	
For $s = 3$, EVLCP~\eqref{eq:evlcp} can be transformed into
\begin{align*}
		\frac{1}{8} \Bigg\{
		& -(4\Omega A_0 + 2A_1 + A_2 + A_3)x - (4\Omega q_0 + 2q_1 + q_2 + q_3)
		+ \big| (A_2 - A_3)x + (q_2 - q_3) \big| \\
		& + \Big| (2A_1 - A_2 - A_3)x + (2q_1 - q_2 - q_3)
		+ \big| (A_2 - A_3)x + (q_2 - q_3) \big| \Big| \\
		& + \bigg| (4\Omega A_0 - 2A_1 - A_2 - A_3)x + (4\Omega q_0 - 2q_1 - q_2 - q_3)
		+ \big| (A_2 - A_3)x + (q_2 - q_3) \big| \\
		&\quad + \Big| (2A_1 - A_2 - A_3)x + (2q_1 - q_2 - q_3)
		+ \big| (A_2 - A_3)x + (q_2 - q_3) \big| \Big| \bigg| \Bigg\} = 0,
	\end{align*}
		where $\Omega \in \mathbb{R}^{n \times n}$ is a positive matrix. Then we have
		\begin{align*}
			B_0 &= -(4\Omega A_0 + 2A_1 + A_2 + A_3 ), b_0 = -(4\Omega q_0 + 2q_1+q_2+q_3),\\
			B_1^1 &= B_2^2 = B_3^2 = B_3^4 = A_2 - A_3,  b_1^1 = b_2^2 = b_3^2 = b_3^4 = q_2 - q_3,\\
			B_2^1 &=  B_3^3 = 2A_1 - A_2 - A_3,  b_2^1 = b_3^3 = 2q_1 - q_2 - q_3,\\
			B_3^1 &= 4\Omega A_0 - 2A_1 - A_2 - A_3,  e = 4\Omega q_0 - 2q_1 - q_2 - q_3.
		\end{align*}
		The dynamical model \eqref{eq:mr}--\eqref{eq:mrd} can be used to solve EVLCP~\eqref{eq:evlcp} with $s=3$ and the condition \eqref{eq:cc} becomes
\begin{equation}\label{ie:cons3}
		\sigma_{\min}\left( -(4\Omega A_0 + 2A_1 + A_2 + A_3) \right) > 4\|A_2 - A_3\| + 2\|2A_1 - A_2 - A_3\| + \|4\Omega A_0 - 2A_1 - A_2 - A_3\|.
\end{equation}
For $s=4$, EVLCP~\eqref{eq:evlcp} can be reduced to
{\small
	\begin{align*}
		\frac{1}{16} \Bigg\{
		& -(8\Omega A_0 + 4A_1 + 2A_2 + A_3 + A_4)x - (8\Omega q_0 + 4q_1 + 2q_2 + q_3 + q_4)
		+ \big| (A_3 - A_4)x + (q_3 - q_4) \big| \\
		& + \Big| (2A_2 - A_3 - A_4)x + (2q_2 - q_3 - q_4)
		+ \big| (A_3 - A_4)x + (q_3 - q_4) \big| \Big| \\
		& + \bigg| (4A_1 - 2A_2 - A_3 - A_4)x + (4q_1 - 2q_2 - q_3 - q_4)
		+ \big| (A_3 - A_4)x + (q_3 - q_4) \big| \\
		&\quad + \Big| (2A_2 - A_3 - A_4)x + (2q_2 - q_3 - q_4)
		+ \big| (A_3 - A_4)x + (q_3 - q_4) \big| \Big| \bigg| \\
		& + \Bigg[ \Bigg| (8\Omega A_0 - 4A_1 - 2A_2 - A_3 - A_4)x + (8\Omega q_0 - 4q_1 - 2q_2 - q_3 - q_4)
		+ \big| (A_3 - A_4)x + (q_3 - q_4) \big| \\
		&\quad + \Big| (2A_2 - A_3 - A_4)x + (2q_2 - q_3 - q_4)
		+ \big| (A_3 - A_4)x + (q_3 - q_4) \big| \Big| \\
		&\quad + \bigg| (4A_1 - 2A_2 - A_3 - A_4)x + (4q_1 - 2q_2 - q_3 - q_4)
		+ \big| (A_3 - A_4)x + (q_3 - q_4) \big| \\
		&\quad\quad + \Big| (2A_2 - A_3 - A_4)x + (2q_2 - q_3 - q_4)
		+ \big| (A_3 - A_4)x + (q_3 - q_4) \big| \Big| \bigg| \Bigg| \Bigg] \Bigg\} = 0,
	\end{align*}}
		where $\Omega \in \mathbb{R}^{n \times n}$ is a positive matrix. Then we have
		\begin{align*}
			B_0 &= -(8\Omega A_0 + 4A_1 + 2A_2 + A_3 +  A_4), b_0 = -(\Omega q_0 +4q_1 + 2q_2 + q_3 + q_4),\\
			B_1^1 &= B_2^2 = B_3^2 = B_3^4 = B_4^2 = B_4^4 = B_4^6 = B_4^8 = A_3 - A_4, \\
			b_1^1 &= b_2^2 = b_3^2 = b_3^4 = b_4^2 = b_4^4 = b_4^6 = b_4^8 = q_3 - q_4,\\
			B_2^1 &= B_3^3 = B_4^3 = B_4^7= 2A_2 - A_3 - A_4,  b_2^1 = b_3^3 = b_4^3 = b_4^7 = 2q_2 - q_3 - q_4,\\
			B_3^1 &= B_4^5 = 4A_1 - 2A_2 - A_3 - A_4, b_3^1 = b_4^5 = 4q_1 - 2q_2 - q_3 - q_4,\\
			B_4^1 &= 8\Omega A_0 - 4A_1 - 2A_2 - A_3 - A_4,  b_4^1 = 8\Omega q_0 - 4q_1 - 2q_2 - q_3 - q_4.
		\end{align*}
		The dynamical model \eqref{eq:mr}--\eqref{eq:mrd} can be used to solve EVLCP~\eqref{eq:evlcp} with $s=4$ and the condition \eqref{eq:cc} becomes
		{\small\begin{align}\nonumber
				\sigma_{\min}&\left( -(8\Omega A_0 + 4A_1 + 2A_2 + A_3 +  A_4) \right) \\\label{ie:cons4}
				&> 8\|A_3 - A_4\| + 4\|2A_2 - A_3 - A_4\| + 2\|4A_1 - 2A_2 - A_3 - A_4\|+ \|8\Omega A_0 - 4A_1 - 2A_2 - A_3- A_4\|.
		\end{align}}
	\end{remark}

\section{Numerical simulations}\label{sec:ne}
	In this section, three numerical examples are given to demonstrate the effectiveness of the dynamical system \eqref{eq:mr}--\eqref{eq:mrd} for solving EVLCP~\eqref{eq:evlcp}. We are interested in the continuous dynamical system and, as a comparison, the neural network proposed in \cite{hozq2022} is adjusted to solve EVLCP~\eqref{eq:evlcp}. All computations are done in MATLAB and the built-in function `ode45' is used to solve the corresponding ordinary differential equations. The integrating range is defined by ${\rm tspan} = [0~t_f]$. For our model, we set $\rho_1 = 200$, $\rho_2 = 100$, $\lambda_1 = 0.5$, $\lambda_2 = 1.5$, $\gamma=100$ and $t_f = T_{\max}$. For the neural network model proposed in \cite{hozq2022}, we set $\tau = \gamma$, $\alpha = 0.0001$ and  $t_f = 0.03$. For all examples,  the initial point is $x_0 = 2\cdot{\rm rand}(n,1) - 1$ with the random seed being $40$. The tested methods are terminated if $t> t_f$ or
$$
		\text{RES} := \left\| \min \left\{ A_0 x + q_0, A_{1} x + q_{1}, A_{2} x + q_{2}, \ldots, A_{s} x + q_{s} \right\} \right\|_{2} \leq  10^{-5}.
$$
In the following numerical results, $T_{\max}$ (which is defined by \eqref{eq:tmax} and only used for our model), RES, the elapsed CPU time in seconds (denoted by ``CPU'') and RERR are reported, where
$$
\text{RERR} := \| x - x_*\|/\|x_*\|
$$
with $x_*$ being the exact solution of EVLCP and $x$ the approximate solution.

	\begin{exam}\label{exams2}
		For $s = 2$, consider EVLCP~\eqref{eq:evlcp} with
		$$
		A_0 = \text{tridiag}\left(0.1,5,-0.2\right)\in \mathbb{R}^{n\times n}, A_1 = I_m \otimes S_1 + \theta I_n, A_{2} = A_1 + S_2\otimes I_m,
		$$
where  $S_{1} = \text{tridiag}(-1.5, 4, -0.5)\in \mathbb{R}^{m\times m}$  and $S_{2} = \text{tridiag}(-1.5, 0, -0.5) \in \mathbb{R}^{m \times m}$. We set $q_{i} = w_{i} - A_{i} x_{*}$ $(i =0, 1, 2)$ with
$$(x_{*})_{(i)}  = \begin{cases} 1, &\text{if} \quad {\rm mod}(i,2) = 1;\\
-0.5,& \text{otherwise}.
\end{cases}$$
Here,
\begin{align*}
(\omega_0)_{(i)} &= \begin{cases} 0, &\text{if} \quad {\rm mod}(i,3) = 0;\\
1,& \text{otherwise},
\end{cases}\\
\omega_1 & = \textbf{1} - \omega_0, \omega_2 = \omega_0 + \omega_1
\end{align*}
with $\textbf{1}$ being the column $n$-vector of all ones.

For this example, we set $\Omega = \frac{1}{2}(D_{A_1} + D_{A_2})$. In our test, we find that \eqref{ie:cons2} holds and thus $x_*$ is the unique solution of this EVLCP. Numerical results for Example~\ref{exams2} with $\theta=2$ are reported in Table~\ref{tab:exams2}. It follows from Table~\ref{tab:exams2} that our model outperforms the model proposed in \cite{hozq2022} in terms of CPU, RES and RERR. Furthermore, we use Figure~\ref{fig:s2} to intuitively demonstrate our claims. Moreover, from Figure~\ref{fig:s2}, we can find that $T_{\max}$ is still conservative. The same phenomenon can be observed in the following Figure~\ref{fig:s3} and Figure~\ref{fig:s4}. For $\theta = 5$, we have similar results which are not reported here in order to save space.
		
\begin{table}[H]
	\centering
	\caption{Numerical results for Example~\ref{exams2} with $\theta = 2$.}\label{tab:exams2}
	\begin{tabular}{ccccccc}
		\toprule
		& $ m $ & $20$ & $40$ & $60$ & $80$ \\ \midrule
		\textbf{Our model} & $T_{\max}$ & 9.0787e-04 & 9.2808e-04 & 9.3215e-04 & 9.3361e-04 \\
		& CPU & 0.1043 & 1.7063 & 11.0343 & 75.4454 \\
		& RES & 1.0000e-05 & 1.0000e-05 & 1.0000e-05 & 1.0000e-05 \\
		& RERR & 1.1989e-07 & 5.7528e-08 & 4.0100e-08 & 2.8856e-08 \\ \midrule
		\textbf{The model in \cite{hozq2022}} & CPU & 3.2565 & 67.6535 & 1389.2136 & 3974.6103 \\
		& RES & 9.5719e-03 & 7.9650e-03 & 3.9306e-03 & 2.8570e-03 \\
		& RERR & 8.3107e-05 & 3.4515e-05 & 1.1359e-05 & 6.1957e-06 \\ \bottomrule
	\end{tabular}
\end{table}
	
		\begin{figure}[H]
			\centering
			\begin{minipage}{0.45\linewidth}
				\centering
				\includegraphics[width=\linewidth]{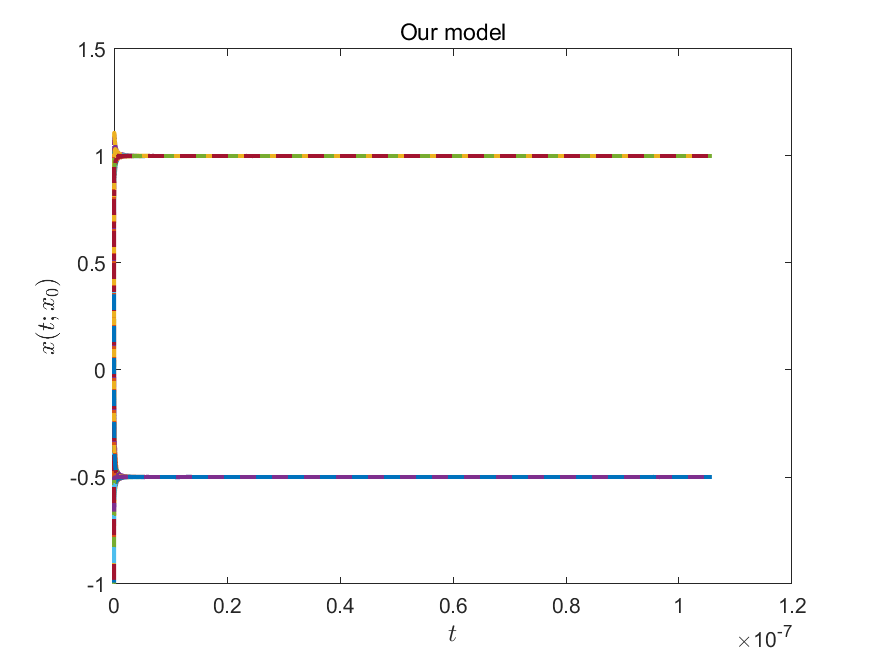}
			\end{minipage} 
			\begin{minipage}{0.45\linewidth}
				\centering
				\includegraphics[width=\linewidth]{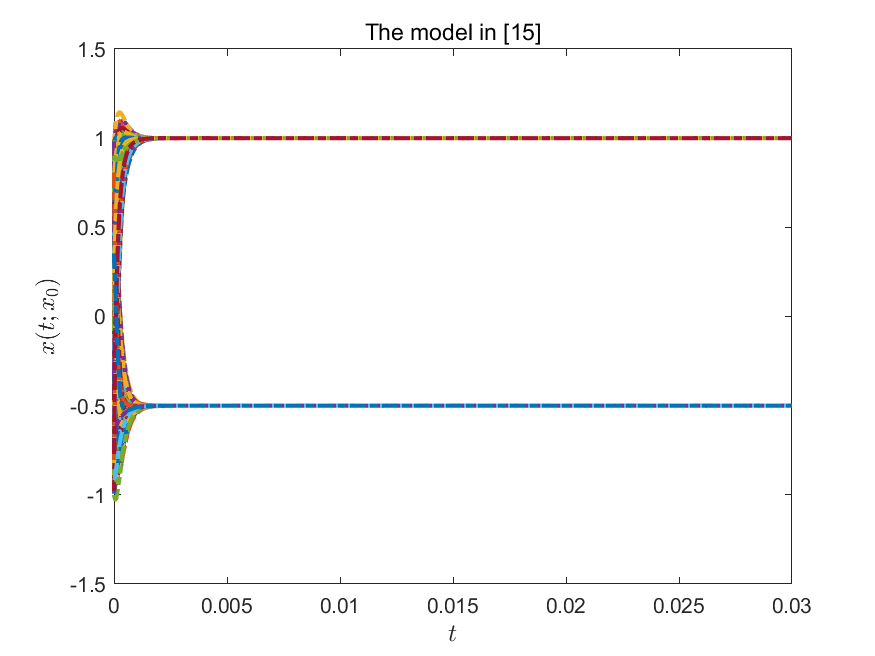}
			\end{minipage} 
			
			\centering
			\begin{minipage}{0.45\linewidth}
				\centering
				\includegraphics[width=\linewidth]{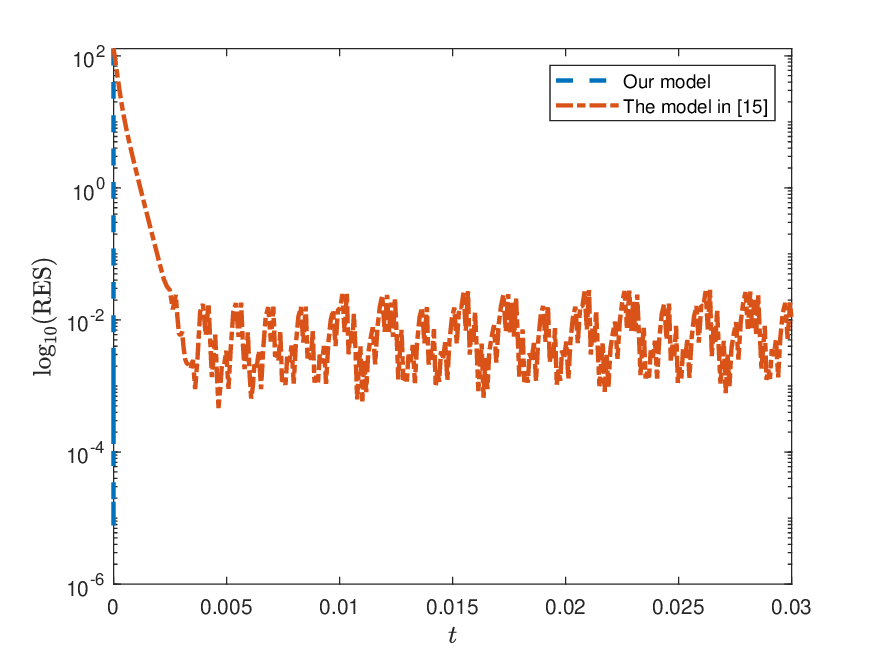}
			\end{minipage} 
			\begin{minipage}{0.45\linewidth}
				\centering
				\includegraphics[width=\linewidth]{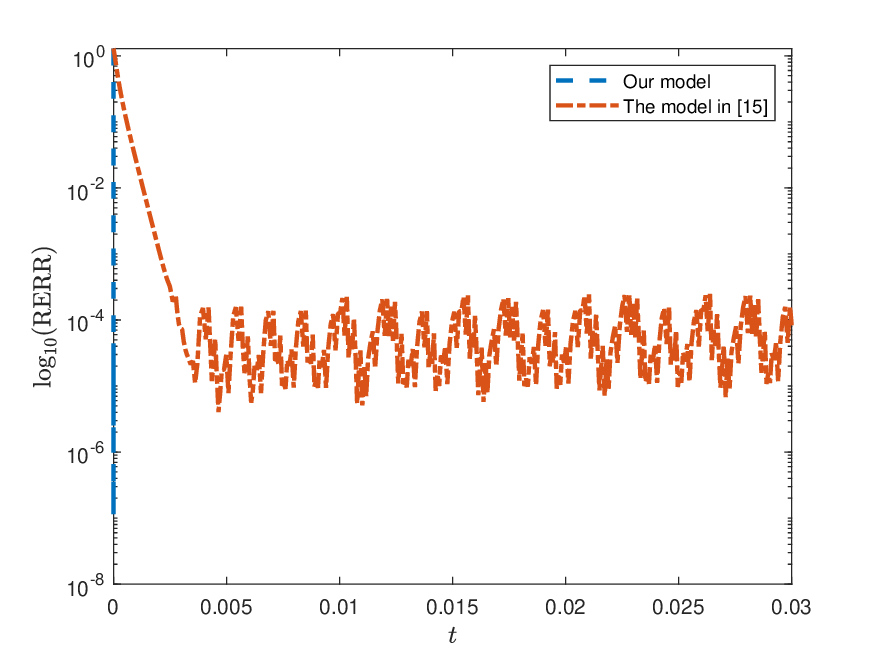}
			\end{minipage}
			\caption{Transient behaviors of $x(t;x_0)$ and curves of RES and RERR for Example~\ref{exams2} with $m = 20$ and $\theta=2$. }\label{fig:s2}
		\end{figure}

\end{exam}

\begin{exam}\label{exams3}
For $s = 3$, consider EVLCP~\eqref{eq:evlcp} with
\begin{align*}
A_0 &= I_n, q_0 = 0,\\
A_1 &= \text{Tridiag}\left(-0.3 I_{m}, S_{1}, -0.2 I_{m}\right),\\
A_2&= \text{Pentadiag} \left(-0.4 I_{m}, O_m, S_{2}, O_m, -0.5 I_{m}\right),\\
A_3&= \text{Pentadiag} \left(-0.5 I_{m}, O_m, S_{2}^\top, O_m, -1.5 I_{m}\right),
\end{align*}
where
$$
S_1 = \text{tridiag}(-0.2, 10, -0.3), S_2 = \text{tridiag}(-0.5, 5.25, -0.4).
$$
In addition, we set $q_{i} = w_{i} - A_{i} x_{*}$ $(i =1, 2, 3)$ with
\begin{equation}\label{eq:zstar}
(x_{*})_{(i)}  = \begin{cases} 1, &\text{if} \quad {\rm mod}(i,2) = 1;\\
0,& \text{otherwise}.
\end{cases}
\end{equation}
Here,
\begin{align*}
(\omega_1)_{(i)} &= \begin{cases} 1-(x_*)_{(i)}, &\text{if} \quad i > \frac{n}{2}-1;\\
1,& \text{otherwise},
\end{cases}\\
(\omega_2)_{(i)} &= \begin{cases} 1-(x_*)_{(i)}, &\text{if} \quad i < \frac{n}{2} + 1;\\
1 - (\omega_1)_{(i)}& \text{otherwise},
\end{cases} \\
\omega_3 &= \omega_1 + \omega_2.
\end{align*}
For this example, we set $\Omega = \frac{2D_{A_{1}} + D_{A_{2}} + D_{A_{3}}}{4}$. In our test, we find that \eqref{ie:cons3} holds and thus $x_*$ is the unique solution of this EVLCP. Numerical results for Example~\ref{exams3} are reported in Table~\ref{tab:exams3}. Similar to Example~\ref{exams2}, Table~\ref{tab:exams3} also tells us that our model is better than the model proposed in \cite{hozq2022} in terms of CPU, RES and RERR, which is intuitively shown in Figure~\ref{fig:s3}.
		
\begin{table}[H]
	\centering
	\caption{Numerical results for Example~\ref{exams3}.}\label{tab:exams3}
	\begin{tabular}{ccccccc}
		\toprule
		& $ m $ & $20$  & $40$  & $60$ & $80$ \\ \midrule
		\textbf{Our model} & $T_{\max}$ & 3.2679e-04 & 3.4262e-04 & 3.4597e-04 & 3.4721e-04 \\
		& CPU & 0.1817 & 3.2085 & 21.6039 & 149.5318 \\
		& RES & 1.0000e-05 & 1.0000e-05 & 1.0000e-05 & 1.0000e-05 \\
		& RERR & 6.9939e-08 & 3.4515e-08 & 5.6533e-08 & 3.8273e-08 \\ \midrule
		\textbf{The model in \cite{hozq2022}} & CPU & 8.1135 & 192.4065 & 1957.9431 & 7880.0633 \\
		& RES & 3.0372e-04 & 6.1105e-04 & 1.0058e-03 & 1.2282e-03 \\
		& RERR & 2.1325e-05 & 2.1579e-05 & 2.1685e-05 & 2.1706e-05 \\ \bottomrule
	\end{tabular}
\end{table}

\begin{figure}[H]
			\centering
			\begin{minipage}{0.45\linewidth}
				\centering
				\includegraphics[width=\linewidth]{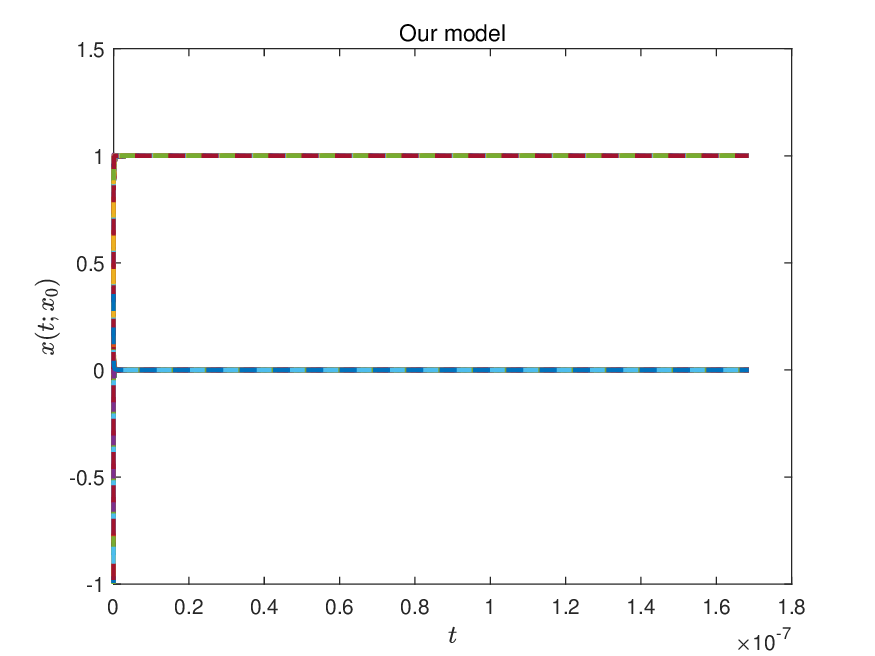}
			\end{minipage} 
			\begin{minipage}{0.45\linewidth}
				\centering
				\includegraphics[width=\linewidth]{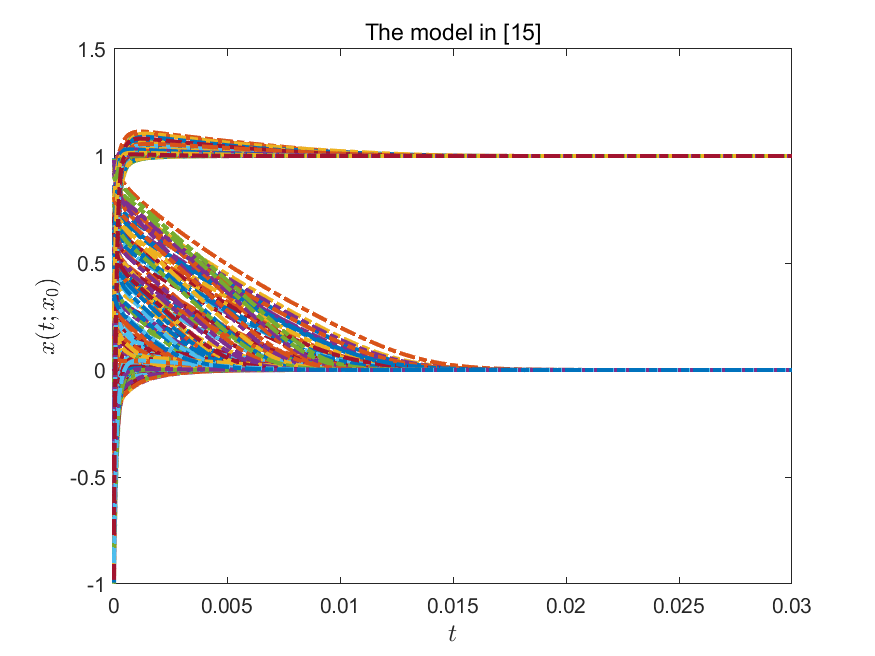}
			\end{minipage} 
			
			\centering
			\begin{minipage}{0.45\linewidth}
				\centering
				\includegraphics[width=\linewidth]{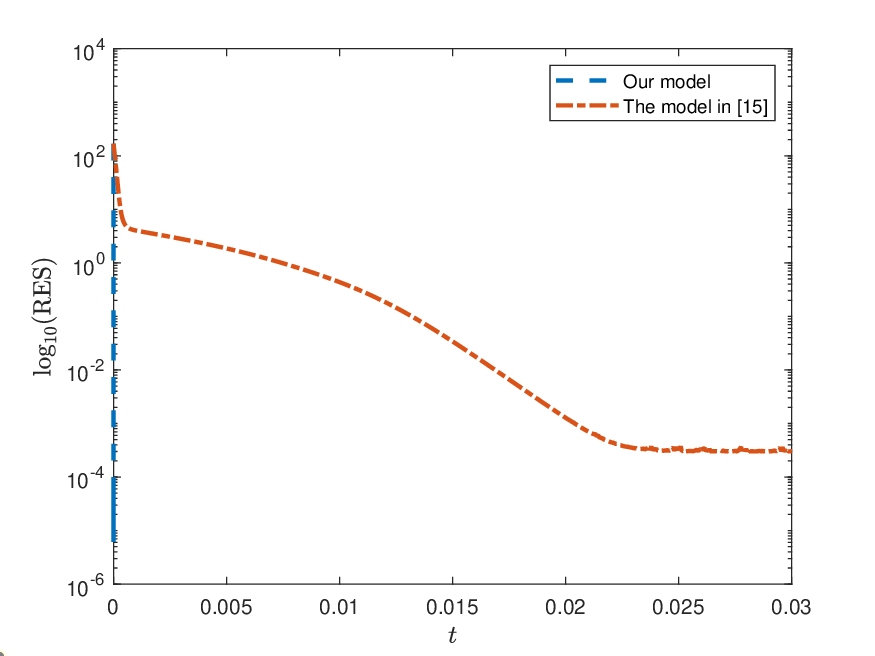}
			\end{minipage} 
			\begin{minipage}{0.45\linewidth}
				\centering
				\includegraphics[width=\linewidth]{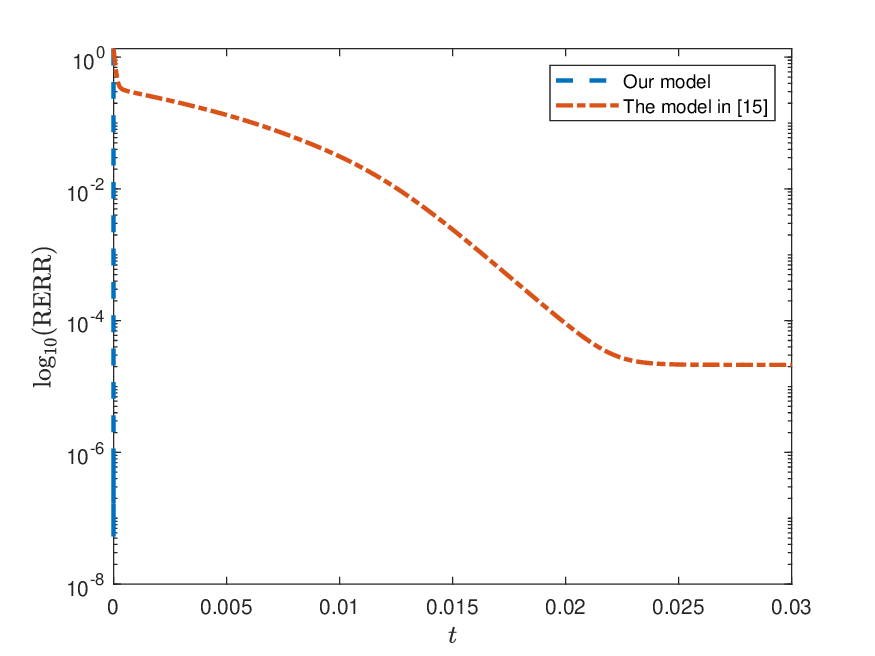}
			\end{minipage}
			\caption{Transient behaviors of $x(t;x_0)$ and curves of RES and RERR for Example~\ref{exams3} with $m = 20$.}\label{fig:s3}
		\end{figure}
		
	\end{exam}

\begin{exam}\label{exams4}
For $s = 4$, consider EVLCP~\eqref{eq:evlcp} with $A_i$ and $q_i$ ($i = 0,1,2,3$) being defined as in Example~\ref{exams3}. In addition, we set $q_4 = w_4 - A_4x_*$ with
$$
A_4 = \text{Pentadiag} \left(-0.2 I_{m}, O_m, S_{3}, O_m, -0.3 I_{m}\right), S_3 = \text{tridiag}(-0.3, 4.8, -0.2)\\
$$
and $w_4 = {\rm rand}(n,1)$. For this example, let $\Omega = \frac{4D_{A_{1}} + 2D_{A_{2}} + D_{A_{3}} + D_{A_{4}}}{8}$. In our test, we find that~\eqref{ie:cons4} holds and thus the EVLCP has the unique solution $x_*$ as defined in \eqref{eq:zstar}. Numerical results for this example are reported in Table~\ref{tab:exams4}, from which we can conclude that our model is superior to the model proposed in \cite{hozq2022} in terms of CPU, RES and RERR. In addition, Figure~\ref{fig:s4} also intuitively demonstrates our claims.

		\begin{table}[H]
	\centering
	\caption{Numerical results for Example~\ref{exams4}.}\label{tab:exams4}
	\begin{tabular}{ccccccc}
		\toprule
		& $ m $ & $20$  & $40$  & $60$ & $80$ \\ \midrule
		\textbf{Our model} & $T_{\max}$ & 2.9045e-04 & 3.1380e-04 & 3.1887e-04 & 3.2073e-04 \\
		& CPU & 0.0277 & 0.0702 & 0.1427 & 0.3906 \\
		& RES & 1.0000e-05 & 1.0000e-05 & 1.0000e-05 & 1.0000e-05 \\
		& RERR & 1.7007e-07 & 8.2353e-08 & 2.8986e-08 & 3.0347e-08 \\ \midrule
		\textbf{The model in \cite{hozq2022}} & CPU & 11.5183 & 246.0174 & 2503.7128 & 12549.4281 \\
		& RES & 3.0140e-04 & 6.1137e-04 & 9.2796e-04 & 1.2289e-03 \\
		& RERR & 2.1312e-05 & 2.1576e-05 & 2.1665e-05 & 2.1705e-05 \\ \bottomrule
	\end{tabular}
\end{table}

\begin{figure}[H]
			\centering
			\begin{minipage}{0.45\linewidth}
				\centering
				\includegraphics[width=\linewidth]{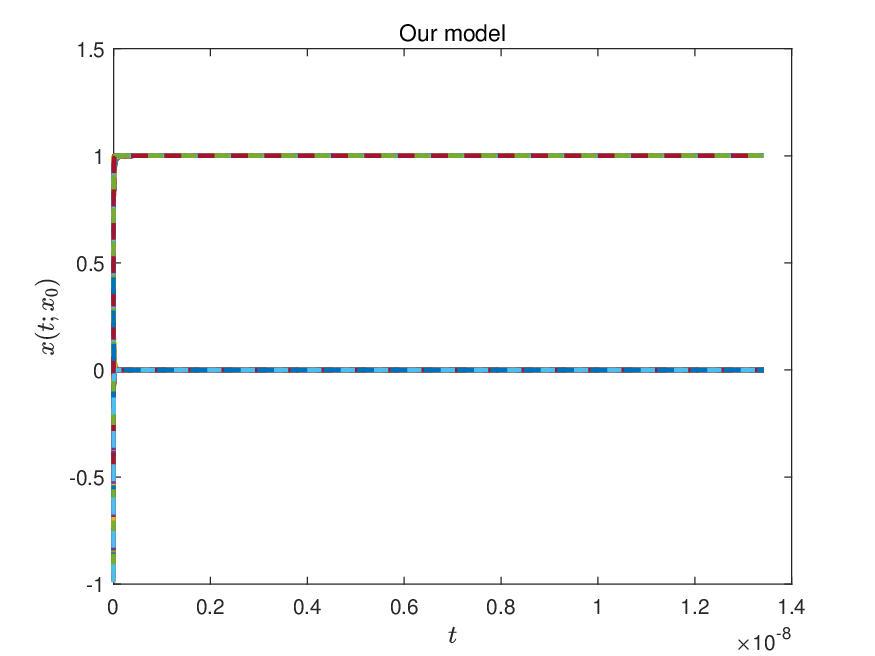}
			\end{minipage} 
			\begin{minipage}{0.45\linewidth}
				\centering
				\includegraphics[width=\linewidth]{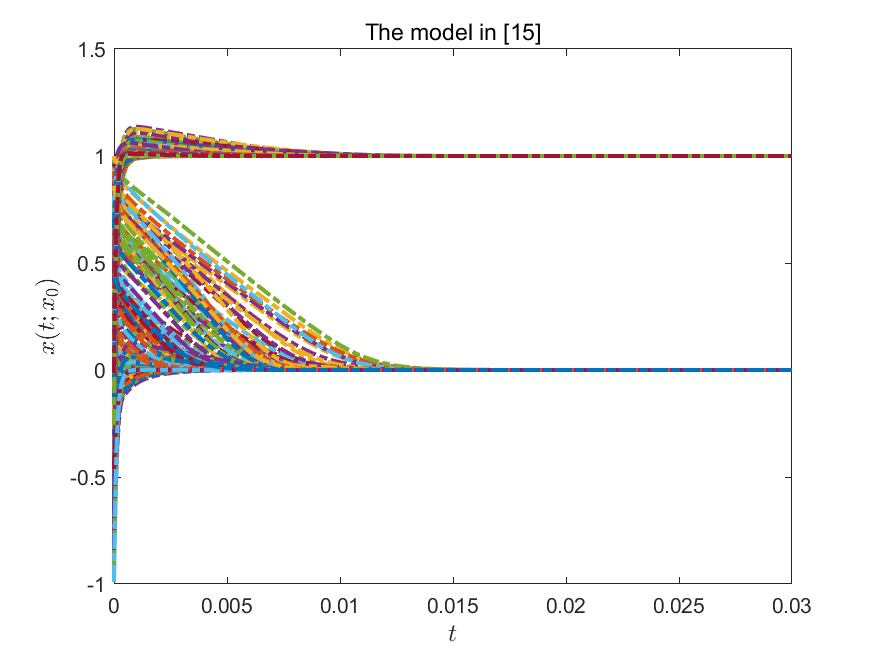}
			\end{minipage} 
			
			\centering
			\begin{minipage}{0.45\linewidth}
				\centering
				\includegraphics[width=\linewidth]{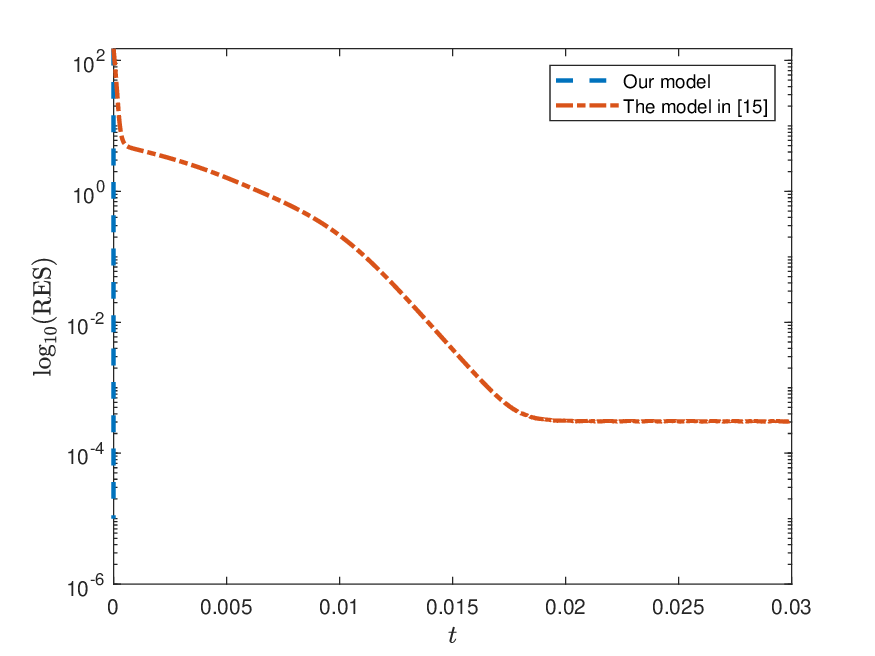}
			\end{minipage} 
			\begin{minipage}{0.45\linewidth}
				\centering
				\includegraphics[width=\linewidth]{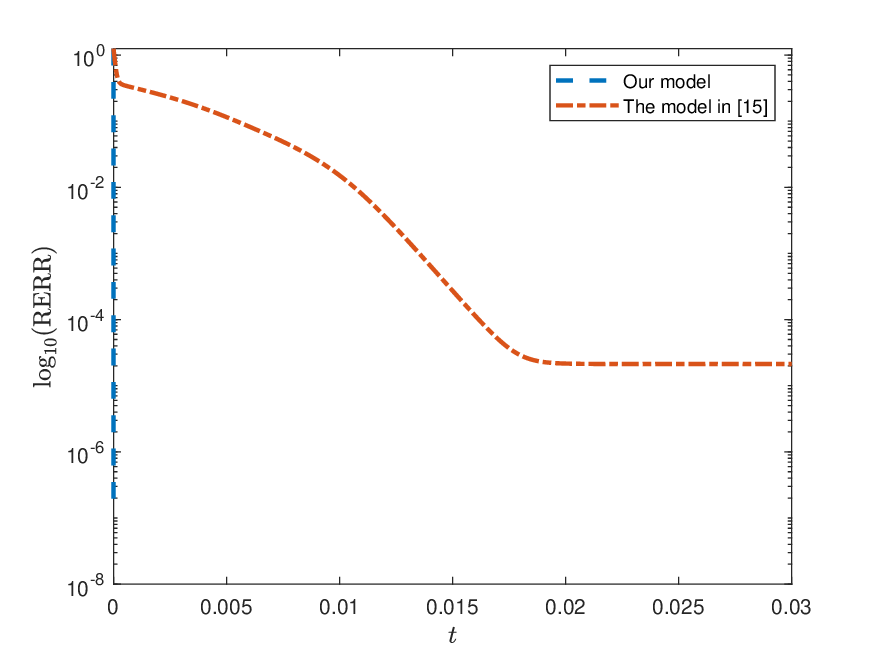}
			\end{minipage}
			\caption{Transient behaviors of $x(t;x_0)$ and curves of RES and RERR for Example~\ref{exams4} with $m = 20$.}\label{fig:s4}
		\end{figure}
		
	\end{exam}
		
\section{Conclusions}\label{sec:con}
A new kind of generalized absolute value equations (NGAVE) is developed and explored in this paper. An error bound and a bound of the residual for the NGAVE are given, which lay the foundation of developing a fixe-time stable dynamical system for solving the NGAVE. In addition, a sufficient condition which guarantees the unique solvability of the NGAVE is also given. As an application, the proposed theory and model are used to solve the extended vertical linear complementarity problem (EVLCP), which is a significant mathematical model in engineering and scientific computing. For solving EVLCP, the neural network proposed in \cite{hozq2022} always performs no better than our model in all our tests, including many not reported here.
	
	
	
\section*{Funding}
The research of Cairong Chen is supported by the Natural Science Foundation of Fujian Province (No. 2025J01673) and the Fujian Alliance of Mathematics
(No. 2023SXLMQN03).  The research of Dongmei Yu is supported by the National Natural Science Foundation of China (No. 12201275), the Ministry of Education in China of Humanities and Social Science Project (No. 21YJCZH204) and the Natural Science Foundation of Liaoning Province (No. 2020-MS-301). The research of Deren Han is supported by the National Natural Science Foundation of China (No. 12131004) and the Ministry of Science and Technology of China (No. 2021YFA1003600).
	
	\section*{Data availability}
	Data sharing not applicable to this paper as no datasets were generated or analyzed during the current study.
	
	\section*{Declarations}
	\subsection*{Competing interests}
	The authors declare no competing interests.
	
	\subsection*{Ethics approval and consent to participate}
	Not Applicable.

	\bibliography{cjcmsample}
	
\end{document}